\documentclass{birkau}
\usepackage{amsmath,amssymb,url}
\usepackage{hyperref}
\usepackage{mathtools}
\usepackage{microtype}
\hypersetup{hidelinks}

\numberwithin{equation}{section}

\theoremstyle{plain}
\newtheorem{theorem}{Theorem}[section]
\newtheorem{lemma}[theorem]{Lemma}
\newtheorem{proposition}[theorem]{Proposition}
\newtheorem{corollary}[theorem]{Corollary}

\theoremstyle{definition}
\newtheorem{definition}[theorem]{Definition}
\newtheorem{example}[theorem]{Example}
\newtheorem{remark}[theorem]{Remark}

\begin{document}

\title[D. Pandey: Metric preservation]{Metric preservation under approximate algebraic laws}
\author[D. Pandey]{Daksh Pandey}
\address{Indian Institute of Technology Roorkee\\
Roorkee, Uttarakhand\\
India}
\email{daksh\_p@ph.iitr.ac.in}

\keywords{Quantitative universal algebra, Free metric algebras, Approximate algebraic identities, Associativity defect, Metric preservation, Quantitative equational theories}

\begin{abstract}
We investigate which metric information can be retained when an initial
metric space is realized in a complete algebra satisfying quantitative
laws. Using free metric realizations, we express preservation of designated
distances as constraints on the permitted errors. In a fixed weighted-path
presentation, the associated feasible error budgets form a closed, convex,
upward-closed region with finite-path certificates of failure, and free
term distances depend continuously on the budgets. For an idempotent binary
operation with positive coordinate weights $a+b=1$, we establish sharp
associativity-defect bounds and an exact formula for the generator distances
in the free realization with associativity tolerance $\alpha$:
$\min\{d_X(x,y),\alpha/(ab)\}$. This yields explicit injectivity and
isometry thresholds. We also give a necessary diameter bound when both
idempotence and associativity are approximate, and a countably presented
example whose free distance is continuous but not Lipschitz at zero.
\end{abstract}

\maketitle

% ============================================================
% 1. INTRODUCTION
% ============================================================

\section{Introduction}
\label{sec:introduction}

\subsection{Motivation and the realization problem}
Mathematical abstraction often proceeds by identifying structural features common to existing mathematical objects. Operations, relations, and axioms are extracted from particular settings and formulated independently of the constructions from which they originally arose. Once formulated, these abstractions allow mathematical reasoning to proceed without committing to a particular realization.

A complementary perspective is to regard structural descriptions not only as abstractions of existing mathematics, but also as starting points for mathematical construction. Given a collection of operations, relations, and desired laws, one may ask what mathematical structures can be constructed to realize those requirements. The specification then serves as the input to a constructive problem, rather than merely as a description of structures already under consideration.

An immediate challenge is determining what should count as a successful realization. The existence of a structure satisfying prescribed laws need not imply that the construction retains mathematical information of interest. In particular, ordinary algebraic identities are always satisfied by a one-element algebra. Consequently, a specification may admit a realization even when that realization is too degenerate to retain meaningful distinctions from an initial mathematical setting.

This observation motivates a more refined realization problem: given a structural specification and an initial mathematical space, can one construct a realization satisfying prescribed requirements while preserving designated information from the original space? When the requirements are incompatible with such preservation, one may further ask whether controlled relaxations of the specification permit nontrivial realizations, and how the necessary loss of information depends on the extent of those relaxations.

\subsection{Preservation under structural specifications}

We study the realization problem for metric
spaces equipped with algebraic operations. Let $(X,d_X)$ be an
initial metric space, and suppose that we wish to construct
a metric algebra $(M,d_M)$ satisfying a prescribed collection
of algebraic laws and quantitative regularity conditions.
The construction is accompanied by a map
\[
    f:X\longrightarrow M,
\]
which relates the original space to its realization.

We require this map to be nonexpansive, so that
\[
    d_M(f(x),f(y))\leq d_X(x,y)
    \qquad (x,y\in X).
\]
Although this condition prevents distances from increasing,
it does not prevent them from decreasing. In particular,
distinct elements of $X$ may be identified in $M$, or their
mutual distances may become strictly smaller. Thus,
nonexpansiveness alone does not ensure that the original
metric information is preserved.

There are several possible levels of preservation.
At the strongest level, one may require $f$ to be an
isometric embedding, satisfying
\[
    d_M(f(x),f(y))=d_X(x,y)
    \qquad (x,y\in X).
\]
A weaker requirement is injectivity, which ensures that
distinct elements remain distinct without necessarily
preserving their distances. More generally, one may
specify selected pairs $(x_i,y_i)$ and require
\[
    d_M(f(x_i),f(y_i))\geq r_i,
\]
where $r_i\geq 0$ represents a prescribed minimum
separation. Such lower bounds allow preservation to
be formulated as an explicit constraint on admissible
realizations.

These preservation requirements may conflict with the
algebraic laws imposed on $M$. For instance, a collection
of identities may be realizable only after certain
distances have contracted or some elements have become
identified. Rather than abandoning the desired algebraic
structure, one may permit controlled violations of
selected identities. An equation $s=t$ can then be
replaced by a quantitative requirement of the form
\[
    d_M(s^M(\mathbf{u}),t^M(\mathbf{u}))
    \leq \varepsilon,
\]
where $\varepsilon\geq 0$ is an allowed error and the
inequality is required to hold for every assignment
$\mathbf{u}$ of the variables occurring in the equation.

This leads to the central quantitative problem considered
in this paper: how do the permitted errors in a structural
specification determine the metric information that can
be preserved in its realizations? In particular, we seek
to characterize preservation obstructions, identify
conditions under which faithful realizations exist,
and determine how these conditions change as the
algebraic requirements are relaxed.

\subsection{Main results and contributions}

Our investigation concerns the relationship between quantitative
algebraic specifications and the metric information retained by
their realizations. We work with bounded metric spaces equipped
with finitary algebraic operations satisfying prescribed
Lipschitz conditions and exact or approximate algebraic laws.
Within this setting, we examine realizability together with
explicit requirements on the distances that a realization must
preserve.

The first part of the investigation concerns the general
structure of preservation constraints. Using free complete
metric algebras and their universal properties, we formulate
preservation requirements in terms of distances between
designated elements or algebraic terms. When the permissible
errors in the algebraic laws are controlled by a finite vector
of nonnegative parameters, we show that the corresponding
term distances are monotone, concave, and continuous functions
of these parameters. Moreover, the set of error budgets
permitting a prescribed collection of lower distance bounds
is closed, convex, and upward closed. Failure of a preservation
requirement admits a finite weighted-path certificate.

We then investigate a concrete incompatibility between
algebraic laws and metric preservation. Consider a binary
operation $m$ satisfying idempotence,
\[
    m(x,x)=x,
\]
and the weighted Lipschitz condition
\[
    d(m(x,y),m(x',y'))
    \leq a\,d(x,x')+b\,d(y,y'),
    \qquad a,b>0,\quad a+b=1.
\]
Writing $A(m)$ for its uniform associativity defect and
$D$ for the diameter of the underlying space, we establish
the bounds
\[
    abD\leq A(m)\leq 2abD.
\]
The lower estimate is attained by ordinary weighted averaging for
all such weights; the upper estimate is attained in the equal-weight
case on a geodesic tripod. In particular, exact associativity is
incompatible with these requirements on any nontrivial metric space.

The principal exact realization result concerns the metric information
that survives when associativity is imposed only approximately. If $\alpha\geq 0$ is the
permitted associativity defect and
$\eta_\alpha:X\rightarrow F_\alpha(X)$ is the canonical map
into the corresponding free complete metric algebra, then
\[
    d_{F_\alpha(X)}
    \bigl(\eta_\alpha(x),\eta_\alpha(y)\bigr)
    =
    \min\left\{
        d_X(x,y),\frac{\alpha}{ab}
    \right\}.
\]
Thus exact associativity forces the free realization to
collapse to a singleton, whereas every positive tolerance
preserves the distinctness of the original points.
Full isometric preservation occurs precisely when
$\alpha\geq ab\,\operatorname{diam}(X)$.

Finally, we investigate the stability of metric realizations
under variations in their permitted error budgets. When
both idempotence and associativity are approximate, we
derive an explicit necessary inequality relating their
defects to the diameter of the underlying space.

We further show that continuity of a free metric with
respect to an error parameter need not imply Lipschitz
stability. Specifically, we construct a countably
presented system of metric constraints for which
\[
    \delta_\varepsilon(a_0,b_0)
    =
    \inf_{n\geq 0}
    \left(2n\varepsilon+q^n\right),
    \qquad 0<q<1,
\]
and consequently,
\[
    \delta_\varepsilon(a_0,b_0)
    =
    \frac{2}{\log(1/q)}
    \varepsilon\log\frac{1}{\varepsilon}
    +O(\varepsilon)
    \qquad (\varepsilon\downarrow 0).
\]
Thus, although the induced distance varies continuously
with the permitted tolerance, its sensitivity can become
unbounded near an exact specification.

These results distinguish the existence of realizations
from their quantitative preservation and stability
properties.

Taken together, the results provide a concrete mathematical
investigation of realization under preservation constraints.
They show how structural requirements can impose unavoidable
losses of metric information, while controlled relaxation
can restore distinctness or even complete preservation of
the original metric.

\subsection{Relation to existing work}

The free-completion machinery used here belongs to the established
theory of quantitative universal algebra. Quantitative algebraic
reasoning replaces exact equations between terms with distance
inequalities and supports universal constructions
\cite{Mardare2016,Bacci2026}. We use these methods as a mathematical
setting for preservation questions; we do not claim a new general
construction of free quantitative algebras.

The connection between quantitative laws and concrete metric
constructions is especially clear in free complete Wasserstein
algebras \cite{Mardare2018}. Our averaging example also touches the
classical theory of barycentric operations
\cite{Stone1949,Neumann1970}. These operations satisfy a system of
parameter-dependent barycentric identities; our obstruction concerns
ordinary associativity of one fixed weighted binary operation.
Idempotent operations with separate metric side inequalities have
also been studied directly by Navascu\'{e}s, Rajan, and Chand
\cite{Navascues2022}. The present question concerns quantitative
associativity obstructions and preservation in free realizations,
rather than the general properties of such operations.

Broader extensions of quantitative algebra include generalized metric
and fuzzy-relational settings \cite{Mio2024}, metric equational
theories \cite{Mardare2025}, and quantitative rewriting
\cite{Gavazzo2023}. Related numerical and affine logical reasoning
is developed in \cite{Bacci2023,Bacci2026Lawvere}. Recent work on
compact quantitative theories studies when quantitative consequences
admit finite derivations \cite{Mio2025Compact,MioRiba2026}. Our
finite-path obstruction certificates concern a different, more
specific assertion: a fixed lower-distance requirement that fails at
a specified budget is refuted by a single finite path. They do not
assert compactness of the entire quantitative theory.

The question developed in this paper is how structural requirements
and their relaxations affect the metric information a realization
can retain. The general feasible-region description, the exact
averaging threshold, and the boundary-stability example address
complementary aspects of that question. Their mathematical role is
separate from the established free-algebra framework on which they
are formulated.

% ============================================================
% 2. MATHEMATICAL FRAMEWORK
% ============================================================

\section{Quantitative algebraic specifications and free realizations}
\label{sec:preliminaries}

We fix a setting in which the realization problem can be posed without
choosing an underlying algebra in advance. The free constructions used in
this section belong to the established theory of quantitative algebras
\cite{Mardare2016,Bacci2026}. Our bounded, weighted setting is a
concrete specialization. We recall an explicit weighted-path
presentation because it makes the preservation arguments in
Section~\ref{sec:preservation} transparent.

Throughout, metric spaces are nonempty and have finite-valued metrics.
Both input spaces and admissible target spaces have diameter at most
one. The diameter restriction sets a common scale for the quantitative
laws; it is not a requirement that the canonical generator map be
isometric. All operation symbols have finite arity, and their coordinate
weights are finite and nonnegative.

\subsection{Metric algebras and quantitative laws}

Let $\Sigma$ be a finitary algebraic signature. For each $n$-ary
operation symbol $\sigma\in\Sigma$, prescribe weights
$w_{\sigma,1},\ldots,w_{\sigma,n}\geq 0$. A \emph{metric
$\Sigma$-algebra} is a metric space $(M,d_M)$ together with an
interpretation $\sigma^M:M^n\to M$ for each symbol in $\Sigma$.
We require the coordinatewise estimate
\begin{equation}\label{eq:weighted-operations}
 d_M\bigl(\sigma^M(u_1,\ldots,u_n),
          \sigma^M(v_1,\ldots,v_n)\bigr)
 \leq \sum_{i=1}^{n}w_{\sigma,i}\,d_M(u_i,v_i).
\end{equation}
In particular, every operation is uniformly continuous. The
weights are part of the specification; in general, their sum need
not equal one.

A \emph{quantitative law} consists of a pair of $\Sigma$-terms
$s_j,t_j$ in finitely many variables and a bound
$\varepsilon_j\geq 0$. Its interpretation in $M$ is the universal
inequality
\begin{equation}\label{eq:quant-law}
 d_M\bigl(s_j^M(\mathbf{u}),t_j^M(\mathbf{u})\bigr)
 \leq \varepsilon_j
 \qquad\text{for every assignment }\mathbf{u}\text{ in }M.
\end{equation}
An ordinary algebraic equation is the case $\varepsilon_j=0$.
The collection of symbols, weights, and laws is denoted by
$\mathcal{S}_{\boldsymbol\varepsilon}$. The notation allows
several different laws to share a tolerance parameter.

\begin{definition}\label{def:admissible}
An \emph{admissible realization} of
$\mathcal{S}_{\boldsymbol\varepsilon}$ is a complete, nonempty
metric $\Sigma$-algebra $(M,d_M)$ of diameter at most one in
which \eqref{eq:weighted-operations} and all instances of
\eqref{eq:quant-law} hold. Given a nonempty metric space
$(X,d_X)$ with $\operatorname{diam}X\leq 1$, an \emph{admissible
generator assignment} is a nonexpansive map $f:X\to M$.
\end{definition}

We also allow \emph{ground constraints} for specified terms
$u,v$ built from named generators in $X$, namely inequalities
\[
 d_M(u^M(f),v^M(f))\leq\gamma.
\]
Unlike a universal law,
such an inequality is required only at the designated generators.
In the presence of ground constraints, the phrase ``admissible
generator assignment'' includes the requirement that they hold.
The distinction between universal substitutions and ground
constraints will be important in the path construction and in
Section~\ref{sec:stability}.

\subsection{Term algebras and free metric constructions}

Write $T=T_\Sigma(X)$ for the set of finite formal terms generated
by the elements of $X$ using the operation symbols of $\Sigma$.
The elements of $X$ are treated as named generators; variables
occurring in universal laws are placeholders to which arbitrary
terms may be substituted.

A \emph{one-hole context} $C[-]$ is a term expression with one
distinguished occurrence of a hole. If the path from that hole to
the root passes through coordinate $i$ of operation $\sigma$,
that step contributes the factor $w_{\sigma,i}$. We denote the
product of these factors by $w(C)$; the empty context has
weight one. Thus the weighted regularity condition implies
\begin{equation}\label{eq:context-bound}
 d_M\bigl(C^M[u],C^M[v]\bigr)
 \leq w(C)\,d_M(u,v)
\end{equation}
for every interpretation of the fixed terms appearing in $C$.

We construct a weighted, undirected graph on $T$. Its elementary
edges, with their assigned costs, are as follows:
\begin{enumerate}
\item Between generators $x,y\in X$, an edge of cost $d_X(x,y)$.
\item Between every pair of terms $s,t\in T$, an edge of cost $1$.
\item For each quantitative law, an edge of cost $\varepsilon_j$
between the terms obtained from $s_j,t_j$ by any substitution of
terms for their variables; each specified ground inequality
supplies an edge of its assigned cost.
\item For every edge $u\leftrightarrow v$ of cost $c$ supplied
by the preceding rules and every one-hole context $C[-]$, an
edge $C[u]\leftrightarrow C[v]$ of cost $w(C)c$.
\end{enumerate}
The last rule may equivalently be applied repeatedly. All paths
below are finite, although the graph itself may be infinite.
Define
\begin{equation}\label{eq:path-distance}
 \delta_{\boldsymbol\varepsilon}(s,t)
 :=\inf\left\{\sum_{e\in P}c(e):
            P\text{ is a finite path from }s\text{ to }t\right\}.
\end{equation}
The empty path from a term to itself has cost zero.

\begin{theorem}[Maximal compatible pseudometric]
\label{thm:largest-pseudometric}
The function $\delta_{\boldsymbol\varepsilon}$ is a pseudometric
on $T$ of diameter at most one. It satisfies the coordinatewise
operation bounds, the generator inequalities
\begin{equation}\label{eq:unit-contracts}
 \delta_{\boldsymbol\varepsilon}(x,y)\leq d_X(x,y),
 \qquad x,y\in X,
\end{equation}
and all the specified laws and ground constraints. Moreover, it
is the largest pseudometric on $T$ with these properties.
\end{theorem}

\begin{proof}
Symmetry follows by reversing paths, the diagonal vanishes by
the empty path, and concatenation of paths gives the triangle
inequality. The cost-one edges force diameter at most one.
Generator and law bounds hold because each is represented by a
single edge.

To verify regularity of operations, fix all but one coordinate of
$\sigma$ and insert each edge of a path from $s$ to $t$ into that
coordinate. The resulting path has cost multiplied by
$w_{\sigma,i}$. Taking the infimum over paths gives
\[
 \delta_{\boldsymbol\varepsilon}
 \bigl(\sigma(\ldots,s,\ldots),
       \sigma(\ldots,t,\ldots)\bigr)
 \leq w_{\sigma,i}\delta_{\boldsymbol\varepsilon}(s,t).
\]
Changing coordinates one at a time proves the full estimate
\eqref{eq:weighted-operations} for terms.

Conversely, let $\rho$ be any pseudometric satisfying the stated
bounds. Each generating edge has $\rho$-distance at most its
assigned cost. By \eqref{eq:context-bound}, the same is true for
its contextual images. The triangle inequality therefore gives
$\rho(s,t)\leq\sum_{e\in P}c(e)$ for every path $P$ from
$s$ to $t$. Taking the infimum proves
$\rho(s,t)\leq\delta_{\boldsymbol\varepsilon}(s,t)$.
\end{proof}

The maximality statement is relevant to preservation: the
construction does not decrease term distances beyond what the
quantitative constraints force. It is therefore natural to use
its distances to measure what a specification permits.

\subsection{Universal properties and canonical maps}

Set $s\sim t$ when $\delta_{\boldsymbol\varepsilon}(s,t)=0$.
The operation bounds show that $\sim$ is a congruence: replacing
any input of an operation by an equivalent term leaves its output
equivalent. Hence $T/\!\sim$ is a metric $\Sigma$-algebra with
uniformly continuous operations. Let
\begin{equation}\label{eq:free-realization}
 F_{\boldsymbol\varepsilon}(X)
 :=\widehat{T/\!\sim}
\end{equation}
be its metric completion, with operations extended by continuity.
Write
\[
 \eta_{\boldsymbol\varepsilon}:X\longrightarrow
              F_{\boldsymbol\varepsilon}(X)
\]
for the canonical generator map. It is nonexpansive by
\eqref{eq:unit-contracts}.

\begin{theorem}[Free complete realization]
\label{thm:free-realization}
The completion $F_{\boldsymbol\varepsilon}(X)$ is admissible.
For every admissible realization $M$ and every admissible
generator assignment $f:X\to M$, there is a unique
nonexpansive $\Sigma$-homomorphism
\[
 \overline f:F_{\boldsymbol\varepsilon}(X)\longrightarrow M,
 \qquad \overline f\circ\eta_{\boldsymbol\varepsilon}=f.
\]
The statement also holds for presentations with ground
constraints, provided $f$ respects them.
\end{theorem}

\begin{proof}
The quotient operations satisfy the required weighted bounds.
They are uniformly continuous and therefore extend uniquely to
operations on the completion. Since every law is a distance
inequality between interpretations of finitary terms, it passes
to the completion by continuity. Thus the completed algebra is
admissible.

For the universal property, evaluate each formal term in $M$
using $f$. Evaluation respects every edge bound: generator edges
by nonexpansiveness, cost-one edges by the diameter bound, law
edges by admissibility, and contextual edges by
\eqref{eq:context-bound}. Consequently,
\begin{equation}\label{eq:evaluation-contracts}
 d_M\bigl(s^M(f),t^M(f)\bigr)
 \leq\delta_{\boldsymbol\varepsilon}(s,t)
 \qquad(s,t\in T).
\end{equation}
Evaluation thus descends to a nonexpansive homomorphism on
$T/\!\sim$ and extends uniquely to its completion, since $M$
is complete. Continuity of the operations ensures that the
extension is a homomorphism. Density of the term quotient
implies uniqueness.
\end{proof}

\begin{corollary}[Exact semantic distances]
\label{cor:semantic-distances}
For all $s,t\in T$,
\begin{equation}\label{eq:semantic-distance}
 \delta_{\boldsymbol\varepsilon}(s,t)
 =\sup_{(M,f)}d_M\bigl(s^M(f),t^M(f)\bigr),
\end{equation}
where $M$ ranges over admissible complete realizations and $f$
over admissible nonexpansive generator assignments. The supremum
is attained simultaneously for every term pair by
$(F_{\boldsymbol\varepsilon}(X),\eta_{\boldsymbol\varepsilon})$.
\end{corollary}

\begin{proof}
Inequality \eqref{eq:evaluation-contracts} gives one direction.
The free realization itself is among the models in the supremum,
and its term distances equal $\delta_{\boldsymbol\varepsilon}$.
\end{proof}

The free construction and semantic principle above are
instances of established quantitative-algebraic methods
\cite{Mardare2016,Bacci2026}. The questions that follow concern
their consequences for designated metric information and for
varying quantitative specifications. In particular, the next section
treats lower-distance requirements as additional conditions on the
realization, rather than as laws used to define the free algebra.

% ============================================================
% 3. PRESERVATION
% ============================================================

\section{Metric preservation and the geometry of relaxation}
\label{sec:preservation}

We now regard preservation requirements as additional data of
the realization problem. Quantitative laws bound some distances
\emph{above}; preservation requirements constrain selected
distances \emph{below}. Whether the two kinds of bounds are
compatible is a question about the specification as a whole.

\subsection{Preservation requirements}

Fix a family of pairs of terms $(u_i,v_i)$ and thresholds
$r_i\geq 0$, indexed by a set $I$. A realization with a generator
assignment $f$ \emph{preserves} these protected distances when
\begin{equation}\label{eq:protected-distances}
 d_M\bigl(u_i^M(f),v_i^M(f)\bigr)\geq r_i
 \qquad(i\in I).
\end{equation}
The special case $u_i=x_i$, $v_i=y_i\in X$ records protected
generator distances. Choosing $r_i=d_X(x_i,y_i)$ expresses
isometry on the chosen pairs, because $f$ is already
nonexpansive. Protecting all generator pairs in this way is
equivalent to the existence of an isometric generator
assignment. In contrast, injectivity requires positive
separation of each distinct pair but does not demand a common
positive threshold.

\begin{proposition}[Simultaneous preservation criterion]
\label{prop:preservation-criterion}
There is an admissible complete realization satisfying all the
requirements \eqref{eq:protected-distances} if and only if
\begin{equation}\label{eq:preservation-criterion}
 \delta_{\boldsymbol\varepsilon}(u_i,v_i)\geq r_i
 \qquad\text{for every }i\in I.
\end{equation}
Whenever these conditions hold, the single free realization
$F_{\boldsymbol\varepsilon}(X)$ witnesses all of them at once.
\end{proposition}

\begin{proof}
Every admissible evaluation is nonexpansive for
$\delta_{\boldsymbol\varepsilon}$ by
Corollary~\ref{cor:semantic-distances}, so the inequalities
\eqref{eq:preservation-criterion} are necessary. Conversely,
the canonical interpretation in the free realization attains
every distance $\delta_{\boldsymbol\varepsilon}(u_i,v_i)$,
so the same free object satisfies all the protected bounds.
\end{proof}

This simultaneous attainment is useful: an optimization over
admissible models does not require a different witness for each
protected pair. It also makes the metric-preservation question
precise without choosing a special concrete target space.

\subsection{Quantitative realizability}

Suppose from now on that the variable tolerances are controlled
by a finite vector
\[
 \boldsymbol\varepsilon=(\varepsilon_1,\ldots,\varepsilon_k)
 \in[0,\infty)^k.
\]
The signature, coordinate weights, input metric, and the
non-variable constraints remain fixed as
$\boldsymbol\varepsilon$ changes. A law instance with tolerance
$\varepsilon_j$ contributes that parameter multiplied by the
weight of any context in which the instance occurs. Hence each
finite path $P$ in the term graph has cost
\begin{equation}\label{eq:affine-path-cost}
 \operatorname{cost}_{\boldsymbol\varepsilon}(P)
 =c_P+\boldsymbol\beta_P\cdot\boldsymbol\varepsilon,
 \qquad c_P\geq 0,\quad
 \boldsymbol\beta_P\in[0,\infty)^k.
\end{equation}
The coefficients depend on the fixed presentation and the
chosen path, but not on the numerical budget values. Thus
\begin{equation}\label{eq:infimum-affine}
 \delta_{\boldsymbol\varepsilon}(s,t)
 =\inf_{P:s\rightsquigarrow t}
   \bigl(c_P+\boldsymbol\beta_P\cdot
                     \boldsymbol\varepsilon\bigr).
\end{equation}

\begin{proposition}[Concavity and monotonicity]
\label{prop:monotone-concave}
For every fixed term pair $s,t$, the function
$\boldsymbol\varepsilon\mapsto
\delta_{\boldsymbol\varepsilon}(s,t)$ is coordinatewise
nondecreasing and concave on $[0,\infty)^k$.
\end{proposition}

\begin{proof}
Every affine path cost in \eqref{eq:infimum-affine} is
coordinatewise nondecreasing, since
$\boldsymbol\beta_P\geq 0$. Their infimum is therefore also
nondecreasing. For $0\leq\theta\leq1$, every path satisfies
\begin{align*}
 &c_P+\boldsymbol\beta_P\cdot
       \bigl(\theta\boldsymbol\varepsilon
            +(1-\theta)\boldsymbol\varepsilon'\bigr)\\
 &\qquad=\theta\bigl(c_P+\boldsymbol\beta_P\cdot
                                      \boldsymbol\varepsilon\bigr)
 +(1-\theta)\bigl(c_P+\boldsymbol\beta_P\cdot
                                      \boldsymbol\varepsilon'\bigr)\\
 &\qquad\geq
 \theta\delta_{\boldsymbol\varepsilon}(s,t)
 +(1-\theta)\delta_{\boldsymbol\varepsilon'}(s,t).
\end{align*}
Taking the infimum over $P$ yields concavity.
\end{proof}

\subsection{Geometry of admissible error budgets}

Let $\mathcal R$ denote the set of budgets for which some
complete admissible realization satisfies all of the protected
bounds \eqref{eq:protected-distances}. By
Proposition~\ref{prop:preservation-criterion},
\[
 \mathcal R
 =\left\{\boldsymbol\varepsilon\in[0,\infty)^k:
 \delta_{\boldsymbol\varepsilon}(u_i,v_i)\geq r_i
 \text{ for every }i\in I\right\}.
\]

\begin{theorem}[Geometry of preservation budgets]
\label{thm:convex-feasibility}
The set $\mathcal R$ is closed, convex, and upward closed in
$[0,\infty)^k$. More precisely,
\begin{equation}\label{eq:halfspace-description}
 \mathcal R
 =\bigcap_{i\in I}\ \bigcap_{P:u_i\rightsquigarrow v_i}
 \left\{\boldsymbol\varepsilon\in[0,\infty)^k:
 c_P+\boldsymbol\beta_P\cdot\boldsymbol\varepsilon
 \geq r_i\right\}.
\end{equation}
If $\boldsymbol\varepsilon\notin\mathcal R$, then there is a
protected pair and a finite path whose cost is strictly below
its required separation.
\end{theorem}

\begin{proof}
By \eqref{eq:infimum-affine}, the inequality
$\delta_{\boldsymbol\varepsilon}(u_i,v_i)\geq r_i$ holds
exactly when every path from $u_i$ to $v_i$ has cost at least
$r_i$. This proves \eqref{eq:halfspace-description}. Each set
in the intersection is a closed halfspace restricted to the
nonnegative orthant; its normal vector
$\boldsymbol\beta_P$ has nonnegative coordinates.
Consequently the intersection is closed, convex, and upward
closed. If it fails, some path has cost below the corresponding
threshold by the definition of infimum.
\end{proof}

There is also a semantic explanation of convexity. If complete
models $M_1,M_2$ witness preservation at budgets
$\boldsymbol\varepsilon,\boldsymbol\varepsilon'$, equip
$M_1\times M_2$ with the metric
\[
 d_\theta\bigl((x_1,x_2),(y_1,y_2)\bigr)
 =\theta d_1(x_1,y_1)+(1-\theta)d_2(x_2,y_2),
 \qquad 0<\theta<1.
\]
With coordinatewise operations and paired generator
assignments, this is a complete admissible model at budget
$\theta\boldsymbol\varepsilon
+(1-\theta)\boldsymbol\varepsilon'$ preserving every protected
lower bound. This argument uses the freedom to choose a
product carrier; the corresponding set of constraints on one
rigidly fixed carrier need not be convex.

The infinite intersection in
\eqref{eq:halfspace-description} should not be confused with a
finite linear program: there can be infinitely many paths, and
no finite halfspace description or general algorithm is claimed.
For perspective, affine and numerical fragments of quantitative
logic have been studied separately \cite{Bacci2023,Bacci2026Lawvere}.

The dependence of individual distances on tolerances is also
regular, including when some budgets vanish.

\begin{theorem}[Continuity of free distances]
\label{thm:tolerance-continuity}
For every fixed term pair $s,t$, the map
$\boldsymbol\varepsilon\mapsto
\delta_{\boldsymbol\varepsilon}(s,t)$ is continuous on the
entire orthant $[0,\infty)^k$. In particular,
coordinatewise decreasing sequences of budgets have the
corresponding limit of term distances. If every coordinate of
$\boldsymbol\varepsilon$ and $\boldsymbol\varepsilon'$ is at
least $h>0$, then
\begin{equation}\label{eq:interior-lipschitz}
 \left|\delta_{\boldsymbol\varepsilon}(s,t)
       -\delta_{\boldsymbol\varepsilon'}(s,t)\right|
 \leq h^{-1}\lVert\boldsymbol\varepsilon
                 -\boldsymbol\varepsilon'\rVert_\infty.
\end{equation}
\end{theorem}

\begin{proof}
As an infimum of continuous affine functions,
$\delta_{\boldsymbol\varepsilon}(s,t)$ is upper
semicontinuous. Let
$\boldsymbol\varepsilon^{(n)}\to\boldsymbol\varepsilon$
and fix $0<\theta<1$. Because the number of coordinates is
finite, eventually
$\boldsymbol\varepsilon^{(n)}\geq
 \theta\boldsymbol\varepsilon$ coordinatewise. By
monotonicity and concavity,
\begin{align*}
 \liminf_{n\to\infty}
 \delta_{\boldsymbol\varepsilon^{(n)}}(s,t)
 &\geq\delta_{\theta\boldsymbol\varepsilon}(s,t)\\
 &\geq\theta\delta_{\boldsymbol\varepsilon}(s,t)
     +(1-\theta)\delta_{\boldsymbol 0}(s,t)\\
 &\geq\theta\delta_{\boldsymbol\varepsilon}(s,t).
\end{align*}
Letting $\theta\uparrow1$ proves lower semicontinuity,
including at boundary points. This, together with upper
semicontinuity, proves continuity.

For the Lipschitz estimate, select a path $P$ with
$\operatorname{cost}_{\boldsymbol\varepsilon}(P)
\leq\delta_{\boldsymbol\varepsilon}(s,t)+\xi\leq1+\xi$.
Since all coordinates of $\boldsymbol\varepsilon$ are at
least $h$ and $\boldsymbol\beta_P\geq0$,
\[
 h\sum_j\beta_{P,j}\leq
 \boldsymbol\beta_P\cdot\boldsymbol\varepsilon
 \leq 1+\xi.
\]
Evaluating the same path at $\boldsymbol\varepsilon'$ gives
\[
 \delta_{\boldsymbol\varepsilon'}(s,t)
 \leq\delta_{\boldsymbol\varepsilon}(s,t)+\xi
     +\frac{1+\xi}{h}
         \lVert\boldsymbol\varepsilon'-
                        \boldsymbol\varepsilon\rVert_\infty.
\]
Let $\xi\downarrow0$ and interchange the two budgets to
obtain \eqref{eq:interior-lipschitz}.
\end{proof}

\subsection{Obstruction certificates}

A finite path $P$ from $u_i$ to $v_i$ satisfying
\[
 c_P+\boldsymbol\beta_P\cdot\boldsymbol\varepsilon<r_i
\]
is an explicit certificate that no admissible model at that
budget can preserve $d(u_i,v_i)\geq r_i$. The converse holds
whenever the preservation condition actually fails: the
strict inequality
$\delta_{\boldsymbol\varepsilon}(u_i,v_i)<r_i$ means that
some finite path attains a cost below $r_i$.

\begin{example}\label{ex:one-ground-bound}
Let $X=\{p,q\}$ with $d_X(p,q)=1$, no operations, and a
ground constraint $d(p,q)\leq\gamma$. The free distance is
\[
 \delta_\gamma(p,q)=\min\{1,\gamma\}.
\]
Preserving $d(p,q)\geq r$ for $0<r\leq1$ is possible exactly
when $\gamma\geq r$. If $\gamma<r$, the single ground edge
is already an obstruction certificate.
\end{example}

\begin{proposition}[Fixed-resolution finite-fragment principle]
\label{prop:finite-fragments}
Fix the signature, coordinate weights, generator bounds, and a
family of protected lower bounds. Suppose a possibly infinite
collection of designated ground law instances is imposed, and
every finite subcollection admits a complete admissible model
preserving all the fixed protected bounds. Then the full
collection admits such a model.
\end{proposition}

\begin{proof}
If the full collection were infeasible, the preceding finite-path
criterion would give a path violating one protected bound.
Only finitely many of the designated law instances occur among
the edges of that path, together with their contextual images.
These finitely many instances would already make the chosen
protected bound impossible, contrary to the hypothesis.
\end{proof}

The qualification ``fixed protected bounds'' is substantive.
The conclusion does not follow if one merely assumes that every
finite fragment admits \emph{some} non-singleton model, with no
uniform positive separation. Section~\ref{sec:stability}
provides an explicit illustration. This fixed-threshold statement
should be distinguished from compactness of a quantitative equational
theory in the proof-theoretic sense
\cite{Mio2025Compact,MioRiba2026}.

% ============================================================
% 4. ASSOCIATIVITY OBSTRUCTIONS
% ============================================================

\section{Associativity obstructions in metric averaging}
\label{sec:obstructions}

We next study one structural specification in which exact laws
impose a particularly rigid geometric obstruction. Averaging
operations on normed convex sets motivate the example, but we
do not assume that the underlying space is linear or convex.
Classical barycentric calculus uses a family of weighted
operations and its own parameter-dependent associativity laws
\cite{Stone1949,Neumann1970}. Idempotent binary operations with
separate metric side inequalities have also been investigated in
\cite{Navascues2022}. Here we deliberately ask about
ordinary associativity for a \emph{single fixed} binary
operation; the two notions should not be conflated.

\subsection{Weighted averaging operations}

Let $(M,d)$ be a nonempty metric space. Fix
\begin{equation}\label{eq:weights}
 a,b>0,\qquad a+b=1,
\end{equation}
and let $m:M\times M\to M$ satisfy
\begin{equation}\label{eq:avg-idempotence}
 m(x,x)=x
\end{equation}
and
\begin{equation}\label{eq:avg-lipschitz}
 d\bigl(m(x,y),m(x',y')\bigr)
 \leq a\,d(x,x')+b\,d(y,y').
\end{equation}
No commutativity or completeness hypothesis is imposed in this
section. Define the \emph{uniform associativity defect}
\begin{equation}\label{eq:associativity-defect}
 A(m):=\sup_{x,y,z\in M}
 d\bigl(m(m(x,y),z),m(x,m(y,z))\bigr).
\end{equation}
Thus $A(m)=0$ if and only if $m$ is associative.

\begin{lemma}[Forced division of distance]
\label{lem:division-of-distance}
For $p=m(x,y)$,
\begin{equation}\label{eq:division-of-distance}
 d(x,p)=b\,d(x,y),\qquad
 d(p,y)=a\,d(x,y).
\end{equation}
\end{lemma}

\begin{proof}
By idempotence and the coordinatewise bounds,
\[
 d(x,p)=d(m(x,x),m(x,y))\leq b\,d(x,y),
\]
\[
 d(p,y)=d(m(x,y),m(y,y))\leq a\,d(x,y).
\]
The triangle inequality gives
\[
 d(x,y)\leq d(x,p)+d(p,y)
 \leq(a+b)d(x,y)=d(x,y).
\]
Equality holds throughout, forcing equality in both
individual estimates.
\end{proof}

In particular, the operation inserts a point dividing the
metric distance between $x$ and $y$ in the fixed ratio $b:a$,
even when $M$ has no vector-space structure.

\subsection{Exact structural incompatibilities}

The obstruction to exact associativity already has a simple
qualitative form.

\begin{proposition}[Associative strict contractions]
\label{prop:strict-assoc-collapse}
Suppose $m$ is idempotent and associative. If it is separately
$b$-Lipschitz in the second variable with $b<1$, then
$m(x,y)=x$ for all $x,y$. Similarly, if it is separately
$a$-Lipschitz in the first variable with $a<1$, then
$m(x,y)=y$ for all $x,y$. Consequently, if both coordinate
constants are strictly below one, $M$ is a singleton.
\end{proposition}

\begin{proof}
Fix $x,y$ and put $p=m(x,y)$. Associativity and idempotence
give
\[
 m(x,p)=m(x,m(x,y))=m(m(x,x),y)=p.
\]
Hence
\[
 d(p,x)=d(m(x,p),m(x,x))\leq b\,d(p,x).
\]
Since $b<1$, $p=x$. For the second assertion, use the
symmetric identity
\[
 m(m(x,y),y)=m(x,m(y,y))=m(x,y)
\]
and the first-coordinate contraction. If both conditions
hold, $x=m(x,y)=y$ for arbitrary $x,y$.
\end{proof}

Thus a nontrivial metric space can admit strictly contractive
idempotent averaging, but not ordinary associativity in
addition. The next theorem quantifies the obstruction.

\subsection{Sharp associativity-defect bounds}

\begin{theorem}[Associativity obstruction]
\label{thm:associativity-obstruction}
Let $m$ satisfy \eqref{eq:weights}--\eqref{eq:avg-lipschitz}
on a metric space $M$ of diameter $D<\infty$. Then
\begin{equation}\label{eq:sharp-associativity-bounds}
 abD\leq A(m)\leq2abD.
\end{equation}
The lower bound is attained for every choice of $a,b$ by weighted
averaging on a bounded convex subset of a normed space. The upper
coefficient is attained for $a=b=1/2$ by geodesic midpoint
interpolation on a metric tripod with three equal arms.
\end{theorem}

\begin{proof}
Fix $x,y$ and set $p=m(x,y)$. Evaluate the associator at
$(x,x,y)$:
\begin{align*}
 &d\bigl(m(m(x,x),y),m(x,m(x,y))\bigr)\\
 &\qquad=d\bigl(p,m(x,p)\bigr)
       =a\,d(x,p)
       =ab\,d(x,y),
\end{align*}
where the last two equalities follow from
Lemma~\ref{lem:division-of-distance}. Taking the supremum
over $x,y$ gives $A(m)\geq abD$.

For the upper bound, insert the intermediate term $m(x,z)$
and use the separate Lipschitz bounds:
\begin{align*}
 &d\bigl(m(m(x,y),z),m(x,m(y,z))\bigr)\\
 &\quad\leq d\bigl(m(m(x,y),z),m(x,z)\bigr)
             +d\bigl(m(x,z),m(x,m(y,z))\bigr)\\
 &\quad\leq a\,d(m(x,y),x)+b\,d(z,m(y,z))\\
 &\quad=ab\bigl(d(x,y)+d(y,z)\bigr)
       \leq2abD,
\end{align*}
again by Lemma~\ref{lem:division-of-distance}.

For lower sharpness, let $K$ be a bounded convex subset of a
normed space and define $m(x,y)=ax+by$. Then
\[
 m(m(x,y),z)-m(x,m(y,z))=ab(z-x).
\]
Consequently $A(m)=ab\operatorname{diam}K$.

For upper sharpness with $a=b=1/2$, let $M$ consist of
three segments of length one glued at a common endpoint $o$.
Its diameter is $D=2$. Give it the path metric and let $m$
return the midpoint of the unique geodesic joining its inputs.
This map is idempotent and $1/2$-Lipschitz in each coordinate.
For completeness, label the arms by $1,2,3$, write a point
as $(i,r)$ with $0\leq r\leq1$, and identify all $(i,0)$
with $o$. For fixed $z=(k,t)$, its midpoint with $(i,r)$ is
\[
 m((i,r),(k,t))=
 \begin{cases}
 (k,(r+t)/2),&i=k,\\
 (i,(r-t)/2),&i\ne k,\ r\geq t,\\
 (k,(t-r)/2),&i\ne k,\ r\leq t.
 \end{cases}
\]
Distances on one arm are $|r-s|$ and on distinct arms are
$r+s$. Comparing these formulas for two varying points
$x,y$ shows
$d(m(x,z),m(y,z))\leq\tfrac12d(x,y)$: on the same
arm there is equality, and on different arms the midpoint
distance is at most half the sum of their radii. Symmetry
gives the corresponding estimate in the second coordinate.
For the three distinct outer endpoints $x,y,z$ we have
$m(x,y)=m(y,z)=o$. Hence
\[
 d\bigl(m(m(x,y),z),m(x,m(y,z))\bigr)
 =d\bigl(m(o,z),m(x,o)\bigr)=1
 =2abD.
\]
This attains the upper estimate.
\end{proof}

\begin{remark}\label{rem:associativity-scope}
The theorem concerns ordinary associativity of one fixed
operation $m$. It is not an impossibility theorem for
barycentric or convex algebras in general: the usual system
of weighted averages satisfies parameter-dependent
barycentric identities rather than ordinary associativity of
one fixed weighted operation. Also, the two sharpness examples
need not coincide; only the upper example requires equal
weights in the stated construction.
\end{remark}

% ============================================================
% 5. EXACT PRESERVATION
% ============================================================

\section{Exact preservation and distance truncation}
\label{sec:truncation}

The preceding obstruction forces any realization with a small
associativity defect to have small diameter. We now show that
this is the \emph{only} restriction on the distances between
original generators in the corresponding free complete
realization.

\subsection{The realization problem}

Fix $a,b>0$ with $a+b=1$. For $\alpha\geq0$, consider the
quantitative specification consisting of one binary symbol
$m$, the coordinate bounds \eqref{eq:avg-lipschitz}, exact
idempotence \eqref{eq:avg-idempotence}, and the universal
approximate-associativity requirement
\begin{equation}\label{eq:alpha-law}
 d\bigl(m(m(x,y),z),m(x,m(y,z))\bigr)\leq\alpha
 \qquad(x,y,z\in M).
\end{equation}
For a nonempty metric input $(X,d_X)$ of diameter at most
one, write $F_\alpha(X)$ for the free complete realization
and $\eta_\alpha:X\to F_\alpha(X)$ for its canonical map.

By Theorem~\ref{thm:associativity-obstruction}, every
admissible target $M$ satisfies
\[
 ab\operatorname{diam}M\leq A(m)\leq\alpha.
\]
Because $\eta_\alpha$ is nonexpansive, it follows that
\begin{equation}\label{eq:truncation-upper}
 d_{F_\alpha(X)}\bigl(\eta_\alpha(x),\eta_\alpha(y)\bigr)
 \leq\min\left\{d_X(x,y),\frac{\alpha}{ab}\right\}.
\end{equation}
The substantive question is whether additional losses of
generator distance are unavoidable.

\subsection{The exact distance-truncation theorem}

\begin{theorem}[Exact generator-distance truncation]
\label{thm:distance-truncation}
Let $X$ be any nonempty metric space of diameter at most one,
and let $a,b>0$ satisfy $a+b=1$. For every $\alpha\geq0$
and $x,y\in X$,
\begin{equation}\label{eq:exact-truncation}
 \boxed{\quad
 d_{F_\alpha(X)}\bigl(\eta_\alpha(x),\eta_\alpha(y)\bigr)
 =\min\left\{d_X(x,y),\frac{\alpha}{ab}\right\}.
 \quad}
\end{equation}
\end{theorem}

\begin{proof}
The upper estimate is \eqref{eq:truncation-upper}. We prove
the converse by constructing one admissible complete algebra
in which the prescribed generator distance is attained.

First assume $\alpha>0$ and put $L=\alpha/(ab)>0$.
Define a truncated metric on the underlying set of $X$ by
\begin{equation}\label{eq:truncated-metric}
 d_L(x,y):=\min\{d_X(x,y),L\}.
\end{equation}
Since $L>0$, this separates distinct points. The triangle
inequality follows from
\[
 \min\{r+s,L\}\leq\min\{r,L\}+\min\{s,L\}
 \qquad(r,s\geq0).
\]
Thus $(X,d_L)$ is a metric space, and
$\operatorname{diam}(X,d_L)\leq\min\{1,L\}$.

Consider the Kuratowski distance-function map into the Banach
space $\ell^\infty(X)$:
\begin{equation}\label{eq:kuratowski}
 \iota:X\longrightarrow\ell^\infty(X),
 \qquad \iota(x)(z)=d_L(x,z).
\end{equation}
All functions $\iota(x)$ are bounded. By the reverse triangle
inequality, for every $z\in X$,
\[
 |\iota(x)(z)-\iota(y)(z)|\leq d_L(x,y).
\]
Taking $z=x$ shows that equality is attained in the
supremum norm. Therefore
\begin{equation}\label{eq:kuratowski-isometry}
 \lVert\iota(x)-\iota(y)\rVert_\infty=d_L(x,y).
\end{equation}

Let $K=\overline{\operatorname{conv}}(\iota(X))$ be the
closed convex hull of the embedded generators. It is a
complete metric space. Moreover,
\begin{equation}\label{eq:convex-hull-diameter}
 \operatorname{diam}K\leq\min\{1,L\}.
\end{equation}
Indeed, if $u=\sum_i\lambda_i s_i$ and
$v=\sum_j\mu_j t_j$ are finite convex combinations of
points of $\iota(X)$, then
\[
 u-v=\sum_{i,j}\lambda_i\mu_j(s_i-t_j),
 \qquad
 \lVert u-v\rVert_\infty
 \leq\sum_{i,j}\lambda_i\mu_j
       \lVert s_i-t_j\rVert_\infty
 \leq\operatorname{diam}(X,d_L).
\]
Pass to the closure to obtain \eqref{eq:convex-hull-diameter}.

Define $m:K\times K\to K$ by ordinary weighted averaging,
\[
 m(u,v)=au+bv.
\]
Convexity makes this a well-defined operation. It is
idempotent, satisfies the required coordinatewise Lipschitz
bound, and obeys
\[
 m(m(u,v),w)-m(u,m(v,w))=ab(w-u).
\]
Its associativity defect is consequently
\[
 A(m)=ab\operatorname{diam}K\leq abL=\alpha.
\]
Thus $K$ is a complete admissible realization of the
$\alpha$-specification. The map $\iota$ is nonexpansive
from $(X,d_X)$ because $d_L\leq d_X$. By the free universal
property there is a nonexpansive homomorphism
$\overline\iota:F_\alpha(X)\to K$ extending $\iota$.
For every $x,y\in X$,
\begin{align*}
 d_L(x,y)
 &=\lVert\iota(x)-\iota(y)\rVert_\infty\\
 &=d_K\bigl(\overline\iota(\eta_\alpha(x)),
                 \overline\iota(\eta_\alpha(y))\bigr)\\
 &\leq d_{F_\alpha(X)}\bigl(\eta_\alpha(x),
                                      \eta_\alpha(y)\bigr).
\end{align*}
Together with \eqref{eq:truncation-upper}, this proves
\eqref{eq:exact-truncation} for $\alpha>0$.

If $\alpha=0$, Theorem~\ref{thm:associativity-obstruction}
forces every admissible nonempty target to have diameter zero.
In particular, $F_0(X)$ is a singleton and both sides of
\eqref{eq:exact-truncation} vanish.
\end{proof}

The lower-bound witness is deliberately concrete, but the
construction of $K$ is not intended to identify the full
free algebra with a convex hull. It is a witness showing that
the upper bound on the distances between the \emph{original
generators} is optimal. The free algebra can contain many
additional points represented by algebraic terms and limits.

\subsection{Injectivity and isometric preservation}

\begin{corollary}[Exact preservation thresholds]
\label{cor:preservation-thresholds}
Under the hypotheses of Theorem~\ref{thm:distance-truncation}:
\begin{enumerate}
\item If $\alpha=0$, the free realization is a singleton.
\item If $\alpha>0$, the map $\eta_\alpha$ is injective.
\item The map $\eta_\alpha$ is isometric if and only if
\begin{equation}\label{eq:isometry-threshold}
 \alpha\geq ab\operatorname{diam}X.
\end{equation}
\end{enumerate}
\end{corollary}

\begin{proof}
The first claim was proved in the theorem. If $\alpha>0$,
then
\[
  \min\{d_X(x,y),\alpha/(ab)\}>0
\]
for distinct points, proving injectivity. Equality of all original distances
holds exactly when
$d_X(x,y)\leq\alpha/(ab)$ for every pair, which is
equivalent to \eqref{eq:isometry-threshold}.
\end{proof}

\begin{corollary}[Protected generator distances]
\label{cor:optimal-protected-budget}
Let $\{(x_i,y_i,r_i)\}_{i\in I}$ be a family with
$x_i,y_i\in X$ and $0\leq r_i\leq d_X(x_i,y_i)$.
There exists a complete admissible realization of the
$\alpha$-specification with a nonexpansive generator map $f$
satisfying
\[
 d_M(f(x_i),f(y_i))\geq r_i\qquad(i\in I)
\]
if and only if
\begin{equation}\label{eq:protected-alpha-threshold}
 \alpha\geq ab\sup_{i\in I}r_i,
\end{equation}
where the supremum of the empty family is $0$.
\end{corollary}

\begin{proof}
Every admissible realization has diameter at most
$\alpha/(ab)$ by Theorem~\ref{thm:associativity-obstruction},
which makes \eqref{eq:protected-alpha-threshold} necessary.
Conversely, if this inequality holds, then
$r_i\leq\min\{d_X(x_i,y_i),\alpha/(ab)\}$ for every $i$.
Theorem~\ref{thm:distance-truncation} shows that the single
free realization satisfies all of these bounds simultaneously.
\end{proof}

The corollary makes the connection with the general feasibility
region in Section~\ref{sec:preservation} explicit: for this
particular algebraic specification, even an arbitrary family of
protected generator distances has a one-dimensional, exact
error-budget threshold.

\begin{example}[Midpoint weights]\label{ex:midpoint-threshold}
Let $X=\{p,q\}$ with $d_X(p,q)=0.8$ and take $a=b=1/2$.
Then
\[
 d_{F_\alpha(X)}\bigl(\eta_\alpha(p),\eta_\alpha(q)\bigr)
 =\min\{0.8,4\alpha\}.
\]
For $\alpha=0$ the generators coincide. At $\alpha=0.05$
they are separated by distance $0.2$; at $\alpha=0.2$ the
original distance $0.8$ is completely preserved. No further
increase in $\alpha$ changes the distance between these
generators.
\end{example}

The result displays two different transitions. Any positive
associativity tolerance restores distinctness of all original
generators, while exact preservation of their metric requires
reaching a definite diameter-dependent threshold. These are
different preservation requirements even though both may be
informally called ``faithful realization.''

% ============================================================
% 6. STABILITY
% ============================================================

\section{Stability and multiple relaxations}
\label{sec:stability}

The exact truncation formula depends Lipschitz-continuously
on its single tolerance parameter. For general quantitative
presentations, Theorem~\ref{thm:tolerance-continuity}
establishes continuity but does not give a uniform Lipschitz
bound at zero. In this section we first relax a second law in
the averaging specification and then show by example that
Lipschitz stability can indeed fail at an exact specification.

\subsection{Approximate idempotence}

Let $m:M^2\to M$ satisfy the same coordinatewise Lipschitz
estimate as before, but initially assume only
\begin{equation}\label{eq:approx-weights}
 0\leq a,b<1,
\end{equation}
without requiring $a+b=1$. Replace idempotence by the
uniform defect bound
\begin{equation}\label{eq:approx-idempotence}
 d(m(x,x),x)\leq\epsilon
 \qquad(x\in M),
\end{equation}
and require the uniform associativity defect to satisfy
\begin{equation}\label{eq:approx-associativity}
 A(m)\leq\alpha.
\end{equation}
The two error parameters express different ways of relaxing
an overconstrained specification. We derive a necessary bound
on the diameter of any such realization.

\subsection{Quantitative stability bounds}

\begin{proposition}[Diameter obstruction with two defects]
\label{prop:robust-diameter}
Under \eqref{eq:approx-weights}--\eqref{eq:approx-associativity},
for every $x,y\in M$,
\begin{align}
 (1-b)d(m(x,y),x)&\leq\alpha+(1+a)\epsilon,
 \label{eq:robust-left}\\
 (1-a)d(m(x,y),y)&\leq\alpha+(1+b)\epsilon.
 \label{eq:robust-right}
\end{align}
Consequently, writing $D=\operatorname{diam}M$,
\begin{equation}\label{eq:robust-diameter-general}
 D\leq
 \left(\frac{1}{1-b}+\frac{1}{1-a}\right)\alpha
 +\left(\frac{1+a}{1-b}+\frac{1+b}{1-a}\right)\epsilon.
\end{equation}
In the normalized case $a,b>0$ and $a+b=1$ this reduces to
\begin{equation}\label{eq:robust-diameter-normalized}
 \boxed{\quad abD\leq\alpha+(1+2ab)\epsilon.\quad}
\end{equation}
\end{proposition}

\begin{proof}
Fix $x,y$ and put $p=m(x,y)$. The associativity bound at
$(x,x,y)$ gives
\[
 d\bigl(m(m(x,x),y),m(x,p)\bigr)\leq\alpha.
\]
By the first-coordinate Lipschitz estimate and approximate
idempotence,
\[
 d\bigl(m(m(x,x),y),p\bigr)
 \leq a\,d(m(x,x),x)\leq a\epsilon.
\]
Hence
\[
 d(p,m(x,p))\leq\alpha+a\epsilon.
\]
Using the triangle inequality and the second-coordinate
Lipschitz estimate,
\begin{align*}
 d(p,x)&\leq d(p,m(x,p))
        +d(m(x,p),m(x,x))+d(m(x,x),x)\\
 &\leq\alpha+a\epsilon+b\,d(p,x)+\epsilon,
\end{align*}
which proves \eqref{eq:robust-left}. An analogous argument
using associativity at $(x,y,y)$ yields
\eqref{eq:robust-right}. Now
\[
 d(x,y)\leq d(x,m(x,y))+d(m(x,y),y).
\]
Insert the preceding estimates and take the supremum over
$x,y$ to obtain \eqref{eq:robust-diameter-general}. When
$a+b=1$, substitute $1-b=a$, $1-a=b$, and simplify to get
\eqref{eq:robust-diameter-normalized}.
\end{proof}

\begin{remark}[A necessary frontier, not a classification]
\label{rem:two-parameter-frontier}
When $a=b=1/2$ and $D=1$, the normalized bound becomes
\[
 \frac14\leq\alpha+\frac32\epsilon.
\]
At $\epsilon=0$, Theorem~\ref{thm:associativity-obstruction}
shows that the coefficient of $\alpha$ is sharp. At a
nonzero idempotence budget, however, the displayed inequality
is only necessary; this investigation does not establish that
all budget pairs satisfying it are realizable on diameter-one
spaces. For comparison, weighted averaging on $[0,1]$
provides a diameter-one realization with
$(\epsilon,\alpha)=(0,ab)$, while the constant operation
$m(x,y)=1/2$ on $[0,1]$ is associative and satisfies the
separate Lipschitz bounds, with idempotence defect at most
$1/2$. Determining the full two-parameter feasible region
requires additional work.
\end{remark}

\subsection{Boundary instability}

We now turn to a different presentation in order to isolate
one feature that the averaging example does not exhibit:
continuity of a free distance at tolerance zero does not
imply Lipschitz continuity there. The example has no
operation symbols and uses countably many ground metric
constraints, rather than a finite presentation of universal
algebraic identities. Thus the effect does not require nonlinear
algebraic operations or a complicated signature.

Fix $0<q<1$ and take the generator set
\begin{equation}\label{eq:ladder-generators}
 X=\{a_n,b_n:n\geq0\}.
\end{equation}
Put $d_X(a_n,b_n)=q^n$ and give every other pair of distinct
generators distance one. This defines a metric: among any
three distinct generators at most one pair can have distance
less than one, so the triangle inequality is immediate.
For $\epsilon\geq0$, impose the ground constraints
\begin{equation}\label{eq:ladder-constraints}
 d(a_n,a_{n+1})\leq\epsilon,
 \qquad d(b_n,b_{n+1})\leq\epsilon,
 \qquad n\geq0.
\end{equation}
We write $\delta_\epsilon$ for the resulting free path
pseudometric. The graph can be visualized as two rails with
edges of cost $\epsilon$ between adjacent levels, joined at
level $n$ by a rung of cost $q^n$.

\begin{theorem}[Non-Lipschitz boundary dependence]
\label{thm:ladder-instability}
For the presentation \eqref{eq:ladder-generators}--
\eqref{eq:ladder-constraints},
\begin{equation}\label{eq:ladder-exact}
 \delta_\epsilon(a_0,b_0)
 =\inf_{n\geq0}\bigl(2n\epsilon+q^n\bigr).
\end{equation}
Let $c=\log(1/q)>0$. As $\epsilon\downarrow0$,
\begin{equation}\label{eq:ladder-asymptotic}
 \delta_\epsilon(a_0,b_0)
 =\frac{2}{c}\,\epsilon\log\frac1\epsilon
  +O(\epsilon).
\end{equation}
In particular, $\delta_\epsilon(a_0,b_0)\to0$ but the
function $\epsilon\mapsto\delta_\epsilon(a_0,b_0)$ is not
Lipschitz at $\epsilon=0$.
\end{theorem}

\begin{proof}
For any integer $n\geq0$, follow the $a$-rail from $a_0$
to $a_n$, cross the rung $a_nb_n$, and return along the
$b$-rail to $b_0$. The total cost is
$2n\epsilon+q^n$, giving the upper bound in
\eqref{eq:ladder-exact}. The choice $n=0$ yields cost one.

A path that uses any additional generator edge of cost one
cannot improve this upper bound. Every other path consists
solely of rail edges and rungs. Let $N$ be its greatest
visited level. Starting and ending at level zero entails at
least $2N$ rail traversals, of total cost at least
$2N\epsilon$. The path must also cross from the $a$-rail
to the $b$-rail. Each rung used at a level at most $N$ has
cost at least $q^N$. Thus the total cost is at least
$2N\epsilon+q^N$. Taking the infimum over all paths
proves \eqref{eq:ladder-exact}.

For the asymptotic estimate, write $q^n=e^{-cn}$ and first
minimize the function
\[
 g_\epsilon(u)=2u\epsilon+e^{-cu},\qquad u\geq0,
\]
over real $u$. For sufficiently small $\epsilon>0$, the
minimum occurs at
\[
 u_* =\frac1c\log\frac{c}{2\epsilon},
 \qquad e^{-cu_*}=\frac{2\epsilon}{c}.
\]
Consequently,
\begin{equation}\label{eq:ladder-continuous-min}
 g_\epsilon(u_*)
 =\frac{2\epsilon}{c}
    \left(\log\frac{c}{2\epsilon}+1\right).
\end{equation}
This is a lower bound for the integer infimum. Taking
$n=\lceil u_*\rceil$ increases the linear term by at most
$2\epsilon$, while decreasing the exponential term, so the
integer infimum differs from
\eqref{eq:ladder-continuous-min} by at most $2\epsilon$.
Expansion of the logarithm proves
\eqref{eq:ladder-asymptotic}.

The right-hand side tends to zero as $\epsilon\downarrow0$.
At $\epsilon=0$, the infimum in
\eqref{eq:ladder-exact} is zero because $q^n\to0$. Finally,
\[
 \frac{\delta_\epsilon(a_0,b_0)
             -\delta_0(a_0,b_0)}{\epsilon}
 =\frac{2}{c}\log\frac1\epsilon+O(1)
 \longrightarrow\infty,
\]
which rules out a Lipschitz estimate near zero.
\end{proof}

\begin{remark}[Finite paths and limiting identification]
\label{rem:ladder-limit}
When $\epsilon=0$, each rail is identified along its
zero-cost edges. The distances $q^n$ between the resulting
rail classes tend to zero, so the two classes are also
identified by the path \emph{infimum}. Nevertheless, no
single finite path from $a_0$ to $b_0$ has cost exactly
zero, because every rung has positive cost. Every finite
subset of the rail constraints leaves $a_0$ and $b_0$
at positive distance, although their distance under the
full countable presentation is zero. This is why the
fixed-resolution qualification in
Proposition~\ref{prop:finite-fragments} cannot be replaced
by a bare requirement of nontriviality. It does not, by itself,
settle whether the same non-Lipschitz behavior can occur for
finitely presented universal quantitative theories.
\end{remark}

The exact truncation result of Section~\ref{sec:truncation}
and the ladder example thus exhibit different stability
regimes. In the first, the effect of relaxing associativity
is captured by a simple piecewise-linear distance formula.
In the second, a countable family of otherwise elementary
constraints creates logarithmic sensitivity at a boundary
of the tolerance space.

% ============================================================
% 7. CONCLUSION
% ============================================================

\section{Discussion and further questions}
\label{sec:discussion}

The question that motivated this investigation was broader
than the particular class of metric algebras studied here.
It arose from the possibility of treating abstract structural
descriptions not only as ways of understanding existing
mathematical objects, but also as starting points for
constructing mathematical structures. What became clearer
during the investigation, however, was that the existence
of a realization is only one part of this problem.
Equally important is understanding what a realization
retains from the mathematical setting with which we began.

In the quantitative setting considered here, this
distinction can be made precise. Algebraic laws and
regularity conditions constrain the possible metric
structures, while preservation requirements specify
which distances or distinctions we wish to retain.
The resulting problem is not simply whether a given
collection of laws has a model, but whether those laws
can be realized without sacrificing designated
mathematical information. When this is impossible,
we can investigate the extent to which the requirements
must be relaxed to recover the desired preservation.

The results illustrate several aspects of this
relationship. The geometry of admissible error budgets
provides a general description of preservation
feasibility, while finite weighted paths supply
certificates when particular requirements cannot
be satisfied. The averaging example makes the
tension especially explicit: exact associativity,
together with the prescribed averaging conditions,
forces a complete collapse of the metric structure.
Yet relaxing associativity by any positive amount
restores injectivity, and sufficiently large
tolerances recover the original metric isometrically.
Thus, the passage from an exact specification to
an approximate one can change not merely the
existence of realizations, but the precise amount
of information they preserve.

Another aspect of the investigation concerns how
these realizations respond to changes in the
specification itself. Although free distances depend
continuously on finite collections of error budgets,
the boundary-instability example shows that this
dependence need not be Lipschitz. Even elementary
metric constraints can therefore produce limiting
behavior that is not apparent from individual
finite stages of the construction. This suggests
that preservation and stability should be treated
as related but distinct questions.

Several directions remain to be investigated.
One is the simultaneous relaxation of multiple
algebraic requirements. In the averaging setting,
we obtained a necessary relation between the
idempotence and associativity defects, but an exact
characterization of the corresponding feasible
region remains unresolved here. Another direction
is to identify conditions under which preservation
regions admit finite descriptions or stronger
stability guarantees. It would also be interesting
to understand more generally when sharp obstruction
bounds characterize the geometry of the realizations
that attain them. These are questions arising from
the present investigation; their status relative
to the full existing literature requires separate
consideration.

The present results are confined to a particular
mathematical setting, and they do not establish a
general procedure for constructing arbitrary
mathematical theories from structural specifications.
Nevertheless, these results show how such a
perspective leads to mathematical questions
with precise answers. The free
construction supplies a realization, the preservation
conditions determine what that realization must
retain, and the obstruction and relaxation results
describe what the structural requirements make
possible.

In this sense, the original motivation has not
disappeared, but has taken a more definite
mathematical form. Beginning with structural
requirements leads naturally to questions not
only about what can be constructed, but also
about what must be preserved, what may be lost,
and what compromises are unavoidable.
We regard the investigation of these questions
as one possible direction toward a more
systematic understanding of mathematical
realization and the development of mathematical
structures from abstract specifications.

\section*{Declarations}

\subsection*{AI use declaration}
LLMs were used for editorial purposes in this
manuscript, all outputs were treated as 
suggestions rather than authoritative guidance,
and were inspected by the authors.

% Bibliography generated with the journal-supplied spmpsci.bst,
% then incorporated here as instructed by AUsample.tex.
% Do not use a separate references.bib file in the submission archive.

\end{document}